\documentclass[12pt]{amsart}
\usepackage{color}
\usepackage{verbatim}

\usepackage{eucal}

\begin{document}
\numberwithin{equation}{section}
\newtheorem{cor}{Corollary}[section]
\newtheorem{theorem}[cor]{Theorem}
\newtheorem{prop}[cor]{Proposition}
\newtheorem{lemma}[cor]{Lemma}
\newtheorem*{lemma*}{Lemma}
\theoremstyle{definition}
\newtheorem{defi}[cor]{Definition}
\theoremstyle{remark}
\newtheorem{remark}[cor]{Remark}
\newtheorem{example}[cor]{Example}

\newcommand{\cC}{\mathcal{C}}
\newcommand{\cL}{\mathcal{L}}
\newcommand{\cH}{\mathcal{H}}
\newcommand{\cP}{\mathcal{P}}
\newcommand{\cun}{\cC^{\infty}}
\newcommand{\cz}{{\mathbb C}}
\newcommand{\hz}{{\mathbb H}}
\newcommand{\rz}{{\mathbb R}}
\newcommand{\fm}{f^{-1}}
\renewcommand{\index}{\mathrm{index}}
\newcommand{\N}{\mathbb{N}}
\newcommand{\ov}{\widetilde}
\newcommand{\oD}{\overline{D}}
\newcommand{\px}{\partial_x}
\newcommand{\py}{\partial_y}
\newcommand{\R}{{\mathbb R}}
\newcommand{\Z}{{\mathbb Z}}
\newcommand{\C}{{\mathbb C}}
\newcommand{\SM}{{\mathbb S}}
\newcommand{\Spec}{\operatorname{Spec}}
\newcommand{\supp}{\mathrm{supp}}
\newcommand{\tf}{{\tilde{f}}}\newcommand{\tg}{{\tilde{g}}}
\newcommand{\tih}{\tilde{h}}
\newcommand{\tphi}{\tilde{\phi}}
\newcommand{\vol}{\mathrm{vol}}
\newcommand{\Xe}{X_\epsilon}
\newcommand{\zz}{{\mathbb Z}}

\def\W{\mathrm{W}}
\def\A{\mathrm{A}}
\def\T{\mathrm{T}}

\def\f{\varphi}
\def\e{\varepsilon}
\def\d{\mathrm{d}}
\def\de{\delta}
\def\O{\Omega}
\def\a{\alpha}
\def\b{\beta}
\def\g{\gamma}
\def\la{\langle}
\def\ra{\rangle}
\def\res{|}
\def\SO{{\rm SO}}
\def\Spin{{\rm Spin}}
\def\U{{\rm U}}
\def\Hol{{\rm Hol}}
\def\Ric{{\rm Ric}}
\def\End{{\rm End}}
\def\tr{{\rm tr}}
\def\id{{\rm Id}}
\def\Scal{{\rm Scal}}
\def\vol{{\rm vol}}
\def\can{h_{\SM^2}}
\newcommand{\be}{\begin{equation}}
\newcommand{\ee}{\end{equation}}
\def\beq{\begin{eqnarray*}}
\def\eeq{\end{eqnarray*}}
\def\p{\psi}
\def\.{{\cdot}}
\def\n{\nabla}
\def\nm{\nabla^g}
\def\dt{{\partial_t}}
\def\dz{\partial_z}
\def\dzb{\partial_{\bar z}}

\newcommand{\D}{\mathbb{D}}
\newcommand{\pr}{\partial_r}

\title{Compact lcK manifolds with parallel vector fields}
\author{Andrei Moroianu}
\address{Andrei Moroianu \\ Universit\'e de Versailles-St Quentin \\
Laboratoire de Math\'ema\-tiques \\ UMR 8100 du CNRS\\
45 avenue des \'Etats-Unis\\
78035 Versailles, France }
\email{andrei.moroianu@math.cnrs.fr}
\date{\today}
\begin{abstract}
We show that for $n>2$  a compact locally conformally K\"ahler manifold $(M^{2n},g,J)$ carrying a non-trivial parallel vector field is either Vaisman, or globally conformally K\"ahler, determined in an explicit way by some compact K\"ahler manifold of dimension $2n-2$.
\end{abstract}

\maketitle

\section{Introduction}

A locally conformally K\"ahler (lcK) manifold is a Hermitian manifold $(M,g,J)$ of real dimension $2n\ge 4$ such that around each point, $g$ is conformal to a K\"ahler metric relative to $J$, cf. \cite{do}.

The differentials of the (logarithms of the) conformal factors glue up to a well-defined closed 1-form on $M$ -- called the Lee form of the lcK structure -- which is exact if and only if $(M,g,J)$ is globally conformally K\"ahler. 

Many complex manifolds which for topological reasons do no carry any K\"ahler metric, have compatible lcK metrics. For example the product metric on the Hopf manifold $\SM^1\times \SM^{2n-1}$ (with odd first Betti number) is lcK with respect to the complex structure induced from the identification
$$\SM^1\times \SM^{2n-1}\simeq (\R/\Z\times \SM^{2n-1})\simeq(\R\times \SM^{2n-1})/\Z\simeq (\C^n\setminus\{0\})/\Z.$$
The Lee form of this structure is easily computed to be the length element of $\SM^1$, and is therefore parallel. Compact lcK manifolds with parallel Lee form  are called {\em Vaisman} \cite{va}, and their structure is well-understood: they are mapping tori of automorphisms of Sasakian manifolds cf. \cite{ov}. Moreover, it was recently proved that every compact {\em homogeneous} lcK manifold is Vaisman \cite{gmo}.

In real dimension 4 it is well known that a compact complex manifold carries a compatible K\"ahler metric if and only if its first Betti number is even \cite{bu}, \cite{la}. It was generally believed that every complex surface with odd first Betti number would carry a compatible lcK structure, until Belgun  has shown that some Inoue surfaces do not carry any lcK structure \cite{be}. He also showed that every Hopf surface admits a compatible lcK metric, and classified all Vaisman complex surfaces.

In this paper we address the following question: {\em Are there non-Vaisman compact lcK manifolds which carry a non-trivial parallel 1-form?} It turns out that the answer to this question is positive, and moreover, one can describe the lcK structure of such manifolds in a very explicit way in all dimensions greater than 4 (cf. Theorem \ref{main} below). These manifolds are globally conformally K\"ahler, but the metric is not K\"ahler in general. In dimension 4 this construction still gives examples of non-Vaisman lcK manifolds carrying a parallel 1-form, but we do not know whether these are the only examples. 

A more general problem, which however will not be considered here, would be to describe all compact lcK manifolds with special holonomy (e.g. with reducible holonomy, or whose holonomy group belongs to the Berger list). Note that unlike K\"ahler manifolds, the Riemannian product of lcK manifolds is no longer lcK (at least not in a canonical way). This somehow indicates that the holonomy reduction of a lcK metric is a strong condition, which might lead in general to classification results in the vein of Theorem \ref{main}.

\section{Some preliminaries on lcK manifolds}

As explained in the introduction, a lcK manifold is a Hermitian manifold $(M,g,J)$ of real dimension $2n\ge 4$ carrying an open cover $U_\a$ and real maps $f_\a:U_\a\to\R$ such that $(U_\a,e^{-f_\a}g,J)$ are K\"ahler manifolds. Denoting $\Omega(\cdot,\cdot):=g(J\cdot,\cdot)$ the fundamental form of $M$, the above condition yields 
\be\label{doa}0=\d(e^{-f_\a}\O)=e^{-f_\a}(-\d f_\a\wedge\O+\d \O).\ee
Since the linear map $\Lambda^1M\to \Lambda^3M$ defined by $\sigma\mapsto\sigma\wedge\O$ is injective, \eqref{doa} shows that $\d f_\a=\d f_\b$ on $U_\a\cap U_\b$, so the 1-forms $\d f_\a$ glue together to a closed form $\theta$ on $M$ -- called the {\em Lee form} -- such that $\theta\res_{U_\a}=\d f_\a$ for all $\a$. The Levi-Civita covariant derivatives $\n$ and $\n^\a$ of the conformal metrics $g$ and $e^{-f_\a}g$ on $U_\a$ are related by the well known formula
$$\n_XY=\n^\a_XY+\tfrac12\left(\theta(X)Y+\theta(Y)X-\theta^\sharp g(X,Y)\right),$$
where $\theta^\sharp$ is the vector field dual to $\theta$ via the metric $g$. Using the fact that $\n^\a J=0$ on $U_\a$, we thus obtain:
\beq(\n_XJ)(Y)&=&\n_X(JY)-J(\n_XY)\\&=&\n^\a_X(JY)+\tfrac12\left(\theta(X)JY+\theta(JY)X-\theta^\sharp g(X,JY)\right)\\&&-J\left(\n^\a_XY+\tfrac12\left(\theta(X)Y+\theta(Y)X-\theta^\sharp g(X,Y)\right)\right)\\
&=&\tfrac12\left(\theta(JY)X-\theta^\sharp g(X,JY)-\theta(Y)JX+J\theta^\sharp g(X,Y)\right).
\eeq
Identifying 1-forms with vectors using the metric $g=:\la\cdot,\cdot\ra$, this relation can be equivalently written as 
\be \label{lck1} (\nabla_XJ)Y=\tfrac12\left(\la X,Y\ra J\theta+\theta(JY)X+\la JX,Y\ra \theta-\theta(Y)JX\right)\qquad\forall X,Y\in\T M,
\ee
or else 
\be \label{lck} \nabla_X\Omega=\tfrac12(X\wedge J\theta+JX\wedge\theta)\qquad\forall X\in\T M.
\ee
If $e_i$ denotes a local orthonormal basis of $\T M$ we have $\Omega=\frac12\sum_ie_i\wedge Je_i$, so by \eqref{lck} we immediately get
\be\label{do}\d\Omega=\sum_ie_i\wedge\nabla_{e_i}\Omega=\theta\wedge\Omega,
\ee
which also follows from \eqref{doa}.

\section{Parallel vector fields on lcK manifolds}

Assume throughout this section that that the dimension of $M$ is strictly larger than $4$ and that $V$ is a non-trivial parallel vector field on $M$. We can of course rescale $V$ such that it has unit length. 
Consider the components of $\theta$ along $V$ and $JV$:
\be \label{a,b}a:=\theta(V),\qquad b:=\theta(JV).
\ee
Since $\nabla V=0$ we have $\nabla_X(JV)=(\nabla_XJ)V$, so using \eqref{lck} we get
\be\label{njv}\nabla_X(JV)=\tfrac12\left(\la X,V\ra J\theta+bX+\la JX,V\ra\theta-aJX\right)\qquad\forall X\in\T M.
\ee
In particular we have
\be\label{nvjv}\nabla_V(JV)=\tfrac12\left(J\theta+bV-aJV\right).
\ee
We also infer from \eqref{njv}
\be\label{djv} \d (JV)=e_i\wedge\nabla_{e_i}(JV)=\tfrac12 V\wedge J\theta-\tfrac12 JV\wedge\theta-a\Omega,
\ee
whence using \eqref{do} and \eqref{djv}:
\beq 0&=& \d^2(JV)=-\tfrac12V\wedge\d(J\theta)-\tfrac12\d(JV)\wedge\theta-\d a\wedge\Omega-a\theta\wedge\Omega\\
&=&-\tfrac12V\wedge\d(J\theta)-\tfrac14 V\wedge J\theta\wedge\theta+\tfrac12 a\Omega\wedge\theta-\d a\wedge\Omega-a\theta\wedge\Omega\\
&=&-\tfrac12V\wedge\left(\d(J\theta)+\tfrac12 J\theta\wedge\theta\right)-\left(\d a+\tfrac12 a\theta\right)\wedge\Omega.
\eeq
Taking the exterior product with $V$ in this relation yields
$$V\wedge\left(\d a+\tfrac12 a\theta\right)\wedge\Omega=0,$$
and since by assumption $n>2$ we get $V\wedge\left(\d a+\tfrac12 a\theta\right)=0$, so there exists some function $f$ on $M$ such that
\be\label{da} \d a+\tfrac12 a\theta=fV.
\ee
Since $\theta$ is closed and $V$ is parallel, the Kostant formula yields
\be\label{nvt}\nabla_V\theta=\cL_V\theta=\d(V\lrcorner\theta)=\d a=fV-\tfrac12 a\theta,
\ee
and a direct computation using \eqref{lck1} gives
\be\label{nvjt}\nabla_V(J\theta)=\tfrac12b\theta+\left(f-\tfrac12|\theta|^2\right) JV.
\ee
Since $V$ is parallel we have $R_{V,X}=0$ for every vector field $X$, where $R_{X,Y}:=[\nabla_X,\nabla_Y]-\nabla_{[X,Y]}$ denotes the curvature tensor of $\nabla$. Consequently, taking $X$ to be $\nabla$-parallel at some point $x\in M$, we obtain
\be\label{nvx}\nabla_V\nabla_X\Omega-\nabla_X\nabla_V\Omega=R_{V,X}\Omega=0\ee
at $x$. Using \eqref{lck1}, \eqref{lck}, \eqref{nvt} and \eqref{nvjt} we compute at $x$:
\beq 2\nabla_V\nabla_X\Omega&=&\nabla_V(X\wedge J\theta+JX\wedge\theta)=X\wedge\nabla_V(J\theta)+JX\wedge\nabla_V\theta+(\nabla_VJ)X\wedge\theta\\
&=&X\wedge\left(\tfrac12b\theta+\left(f-\tfrac12|\theta|^2\right) JV\right)+JX\wedge\left(fV-\tfrac12 a\theta\right)\\
&&+\tfrac12\left(\la V,X\ra J\theta+\theta(JX)V+\la JV,X\ra \theta-\theta(X)JV\right)\wedge\theta
\eeq
and 
\beq 2\nabla_X\nabla_V\Omega&=&\nabla_X(V\wedge J\theta+JV\wedge\theta)=V\wedge\nabla_X(J\theta)+JV\wedge\nabla_X\theta+(\nabla_XJ)V\wedge\theta\\
&=&V\wedge\nabla_X(J\theta)+JV\wedge\nabla_X\theta\\
&&+\tfrac12\left(\la X,V\ra J\theta+bX+\la JX,V\ra \theta-aJX\right)\wedge\theta
\eeq
After straightforward simplifications we get from \eqref{nvx}:
\beq 0&=& \nabla_V\nabla_X\Omega-\nabla_X\nabla_V\Omega\\
&=&\left(\nabla_X(J\theta)+fJX-\tfrac12\theta(JX)\theta\right)\wedge V+\left(\nabla_X\theta+\left(f-\tfrac12|\theta|^2\right)X+\tfrac12\theta(X)\theta\right)\wedge JV.
\eeq
This relation is tensorial in $X$, so it actually holds at every point of $M$.

Remark now that if $A\wedge V+B\wedge JV=0$ for some vectors $A$ and $B$, then both vectors belong to the plane generated by $V$ and $JV$. The previous relation thus shows that there exist some 1-forms $\mu$ and $\nu$ such that 
\be\label{nt}\nabla_X\theta+\left(f-\tfrac12|\theta|^2\right)X+\tfrac12\theta(X)\theta=\mu(X)V+\nu(X)JV,\qquad\forall X\in\T M.
\ee
We take the exterior product with $X$ in this relation and sum over some local orthonormal basis $X=e_i$. As $\d\theta=0$, we get $\mu\wedge V+\nu\wedge JV=0$, hence by the previous remark there exist smooth functions $\a,\ \b,\  \g$ on $M$ such that $\mu=\a V-\g JV$ and $\nu=\g V+\b JV$. Taking $X=V$ in Equation \eqref{nt} and using \eqref{nvt} yields
$$fV-\tfrac12 a\theta+\left(f-\tfrac12|\theta|^2\right)V+\tfrac12a\theta=\a V +\g JV,$$
whence $\g=0$ and 
\be\label{alpha}\a=2f-\tfrac12|\theta|^2.\ee 
Equation \eqref{nt} thus becomes 
\be\label{nt1}\nabla_X\theta=\left(f-\a\right)X-\tfrac12\theta(X)\theta+\a\la X,V\ra V+\b\la X,JV\ra JV,\qquad\forall X\in\T M.
\ee
Using this relation together with \eqref{lck1} we readily obtain
\be\label{njt}\nabla_X(J\theta)=-fJX+\tfrac12\theta(JX)\theta+\a\la X,V\ra JV-\b\la X,JV\ra V,\qquad\forall X\in\T M.
\ee
In particular the exterior derivative of $J\theta$ reads
\be\label{djt}\d (J\theta)=e_i\wedge\nabla_{e_i}(J\theta)=-2f\Omega+\tfrac12\theta\wedge J\theta+(\a+\b)V\wedge JV.
\ee
We now take the scalar product with $V$ in \eqref{njt} and obtain
\be\label{nxjt}\la\nabla_X(J\theta),V\ra=f\la X,JV\ra+\tfrac12a\theta(JX)-\b\la X,JV\ra ,\qquad\forall X\in\T M.\ee
On the other hand 
$$\la\nabla_X(J\theta),V\ra=\nabla_X\la J\theta,V\ra=-X(b),$$
so fromh \eqref{nxjt} we obtain
\be\label{db}\d b=(\b-f)JV+\tfrac12aJ\theta.\ee
Taking the exterior derivative in this equation and using \eqref{djv}, \eqref{da} and \eqref{djt} yields
\beq 0&=& \d^2 b=\d(\b-f)\wedge JV+(\b-f) (\tfrac12V\wedge J\theta-\tfrac12 JV\wedge\theta-a\Omega)\\
&&+\tfrac12\left(fV-\tfrac12 a\theta\right)\wedge J\theta+\tfrac12a\left(-2f\Omega+\tfrac12\theta\wedge J\theta+(\a+\b)V\wedge JV\right)\\
&=&\d(\b-f)\wedge JV+\tfrac12\beta V\wedge J\theta-\tfrac12(\b-f) JV\wedge\theta+\tfrac12a(\a+\b)V\wedge JV-a\beta\Omega.
\eeq
This shows in particular that $V\wedge JV\wedge(a\beta\Omega)=0$, whence 
\be\label{ab}a\beta=0.\ee 
Reinjecting in the previous equation yields
\be\label{dbf}\d(\b-f)\wedge JV+\tfrac12\beta V\wedge J\theta-\tfrac12(\b-f) JV\wedge\theta+\tfrac12a\a V\wedge JV=0.\ee
We now use \eqref{djt} together with \eqref{do} and \eqref{djv}:
\beq 0&=&\d^2(J\theta)=-2\d f\wedge\Omega-2f\theta\wedge\Omega-\tfrac12\theta\wedge\left(-2f\Omega+\tfrac12\theta\wedge J\theta+(\a+\b)V\wedge JV\right)\\&&
+\d(\a+\b)\wedge V\wedge JV-(\a+\b)V\wedge\left(\tfrac12 V\wedge J\theta-\tfrac12 JV\wedge\theta-a\Omega\right)\\
&=&\left(-2\d f-f\theta+a(\a+\b)V\right)\wedge\Omega+\d(\a+\b)\wedge V\wedge JV.
\eeq
As $n>2$, this shows that 
\be\label{df}2\d f+f\theta=a\a V.\ee 
Using this relation together with \eqref{dbf} yields 
$$(\d \b+\tfrac12 \b\theta)\wedge JV+\tfrac12 \b V\wedge J\theta=0.$$
We take the interior product with $V$ in this relation and obtain
$$V(\b)JV+\tfrac12 \b J\theta+\tfrac12b\beta V=0.$$
Since by \eqref{ab}, $JV$ is orthogonal to $\b(J\theta+bV)$, this implies that $V(\beta)=0$ and $\beta(J\theta+bV)=0$.

We now use \eqref{nt1} in order to express the differential of the square norm $|\theta|^2$. For every tangent vector $X$ we have
\beq X(|\theta|^2)&=&2\la \nabla_X\theta,\theta\ra=2\la \left(f-\a\right)X-\tfrac12\theta(X)\theta+\a\la X,V\ra V+\b\la X,JV\ra JV,\theta\ra\\
&=&\la X,2(f-\a)\theta-|\theta|^2\theta+2a\a V+2b\b JV\ra,
\eeq
so from \eqref{alpha} we get 
$$\d |\theta|^2=2(f-\a)\theta-|\theta|^2\theta+2a\a V+2b\b JV=-2f\theta+2a\a V+2b\b JV,$$
whence using \eqref{df}:
\be \label{dalpha} \d\a=\d\left(2f-\tfrac12|\theta|^2\right)=a\a V-f\theta-(-f\theta+a\a V+b\b JV)=-b\b JV.
\ee

We are now ready to prove the key result of this section

\begin{lemma} If $M$ is compact, the Lee form $\theta$ of the lcK structure belongs to the space generated by $V$ and $JV$. Equivalently, 
\be\label{theta}\theta=aV+bJV.\ee
\end{lemma}
\begin{proof} Let $\d\mu_g$ denote the volume form of $M$. Taking the trace in \eqref{njv} we get $\delta (JV)=\frac{2-n}2b$ and from \eqref{nt1} together with \eqref{alpha} we readily compute 
$$\delta\theta=n(\a-f)+\tfrac12|\theta|^2-\a-\b=(n-2)(\a-f)-\beta.$$
Moreover, taking the scalar product with $V$ in \eqref{njt} and choosing $X=JV$ we obtain
$$-JV(b)=\la\nabla_{JV}J\theta,V\ra=f-\tfrac12a^2-\beta,$$
which together with \eqref{da} yields $\beta=JV(b)+V(a)$.
Using the Stokes Theorem several times we obtain
$$\int_Mf\d\mu_g=\int_M \left(V(a)+\frac12 a^2\right)\d\mu_g=\int_M\left(a\delta V+\frac12 a^2\right)\d\mu_g=\frac12\int_M a^2\d\mu_g,$$
and
\beq\int_M(\a-f)\d\mu_g&=&\frac1{n-2}\int_M \left(\beta+\delta\theta\right)\d\mu_g=\frac1{n-2}\int_M\beta\d\mu_g\\&=&\frac1{n-2}\int_M(V(a)+JV(b)\d\mu_g=\frac1{n-2}\int_M(a\delta V+b\delta(JV))\d\mu_g\\&=&\frac1{n-2}\int_M\frac{2-n}2b^2\d\mu_g=-\frac12\int_M b^2\d\mu_g,
\eeq
so finally
\beq \int_M|\theta-aV-bJV|^2\d\mu_g&=&\int_M(|\theta|^2-a^2-b^2)\d\mu_g=\int_M( 4f-2\a-a^2-b^2)\d\mu_g\\
&=&\int_M( 2f-2(\a-f)-a^2-b^2)\d\mu_g=0.
\eeq
\end{proof}

From now on $M$ will be assumed compact.

\begin{lemma} \label{l2} The following relations hold: $ab=0$, $f=\frac{a^2}2$, $\alpha=\frac12(a^2-b^2)$.
\end{lemma}
\begin{proof} 
Taking the covariant derivative in \eqref{theta} with respect to some arbitrary vector X and using \eqref{njv}, \eqref{da} and \eqref{db} yields:
\beq \nabla_X\theta&=&X(a)V+X(b)JV+b\n_XJV\\
&=&f\la X,V\ra V-\tfrac a2\la X,aV+bJV\ra V+(\beta-f)\la X,JV\ra JV+\tfrac a2\la X,aJV-bV\ra JV\\
&&+\tfrac12 b\left(\la X,V\ra(aJV-bV)+bX+\la JX,V\ra(aV+bJV)-aJX\right)\\
&=& \tfrac12b^2X-\tfrac12abJX+\left(f-\tfrac{a^2}2-\tfrac{b^2}2\right)\la X,V\ra V-ab\la X,JV\ra V\\
&&+\left(\beta-f+\tfrac{a^2}2-\tfrac{b^2}2\right)\la X,JV\ra JV.
\eeq
Comparing with \eqref{nt1} we thus get:
\beq \nabla_X\theta&=&\left(f-\a\right)X-\tfrac12\theta(X)\theta+\a\la X,V\ra V+\b\la X,JV\ra JV\\
&=&\left(f-\a\right)X +\tfrac12\la X,aV+bJV\ra(aV+bJV)+\a\la X,V\ra V+\b\la X,JV\ra JV\eeq
and identifying the corresponding terms yields the result.
\end{proof}

Using Lemma \ref{l2} we now get from \eqref{da}:
$$\d a=fV-\tfrac12 a\theta=fV-\tfrac12a^2V-\tfrac12abJV=0,$$
thus showing that $a$ is constant on $M$. We distinguish two cases:

{\bf Case 1: $a\ne 0$.} From Lemma \ref{l2} we must have $b=0$, whence $\theta=aV$ is parallel, so $(M,g,J)$ is Vaisman and the parallel vector field $V$ is proportional to the Lee form.

{\bf Case 2: $a= 0$.} From Lemma \ref{l2} again we get $f=0$, $\alpha=-\frac12b^2$ and $\theta=bJV$. Equation \eqref{njv} now reads
\be\label{njv1}\nabla_X(JV)=\tfrac12b\left(X-\la X,V\ra V-\la X,JV\ra JV\right)\qquad\forall X\in\T M, 
\ee
which by symmetrization gives:
\be\label{njv2}\mathcal{L}_{JV}g=b(g-V^\flat\otimes V^\flat-JV^\flat\otimes JV^\flat).
\ee
\begin{lemma}
The universal cover $(\tilde M,\tilde g,\tilde \Omega)$ of $(M,g,\Omega)$ is (holomorphically) isometric to $\R^2\times (N,g_N,\Omega_N)$ endowed with the metric $\d s^2+\d t^2+e^{2c(t)}g_N$ and the K\"ahler form $\d s\wedge\d t+e^{2c(t)}\Omega_N$, for some complete simply connected K\"ahler manifold $(N,g_N,\Omega_N)$ of real dimension $2n-2$ and some smooth real function $c$.
\end{lemma}
\begin{proof} 
The tangent bundle of $M$ is the direct sum of three orthogonal distributions: $\T M=\la V\ra\oplus \la JV\ra\oplus D$, where $D:=\la V,JV\ra^\perp$. Since $\d V=0$ (as $V$ is parallel) and $\d(JV)=0$ by \eqref{njv1}, the distribution $D$ is involutive. From  \eqref{njv1} again we easily check that $[V,JV]=0$, $[V,D]\subset D$ and $[JV,D]\subset D$.
The Frobenius theorem shows that there exist local coordinates $(s,t,x)\in\R\times\R\times \R^{n-2}$ around each point of $M$, such that $V=\partial_s$, $JV=\partial_t$ and the metric $g$ has the form $g=\d s^2+\d t ^2+h(s,t)$, for some family of metrics $h(s,t)$ on $\R^{n-2}$. 

We first note that for each $s,t$ the metric $h(s,t)$ is K\"ahler. Indeed, $J$ defines by restriction to $D$ an integrable complex structure on each local leaf $\R^{n-2}$, whose K\"ahler form $\Omega(s,t)$ is just the restriction of $\Omega$. Consequently, $\d \Omega(s,t)$ is the restriction to the leaves of $\d \Omega=\theta\wedge\Omega$, which vanishes since $\theta|_D=0$.

Now, since $V$ is parallel on $M$, its flow preserves $g$, whence $h(s,t)=h(t)$ is independent on $s$. From \eqref{db} we see that $b=b(t)$ depends on $t$ alone. Moreover, \eqref{njv2} yields 
$$\frac{\partial h}{\partial t}=b(t)h,$$
whence 
$$h(s,t)=e^{\int_0^tb(\tau) \d\tau}h(0).$$
This proves the local version of the lemma, by defining $c(t):=\tfrac12\int_0^tb(\tau) \d\tau$ and $g_N:=h(0)$. The global statement follows from the Frobenius theorem applied to the universal cover of $M$.
\end{proof} 
The fundamental group of $M$ induces a co-compact group of isometries of the globally conformally K\"ahler manifold $(\tilde M,\tilde g):=(\R^2\times N,\d s^2+\d t^2+e^{2c(t)}g_N)$. Our aim is to show that the Lee form of $M$ is exact. Note that the K\"ahler form of $\tilde M$ is $\tilde\Omega=\d s\wedge\d t+e^{2c(t)}\Omega_N$, which satisfies 
$$\d\tilde\Omega=2c'(t) \d t\wedge e^{2c(t)}\Omega_N=2c'(t) \d t\wedge\tilde\Omega=2\d c\wedge\tilde\Omega,$$
showing that the Lee form of $\tilde M$ is $2\d c$. It suffices to check that the function $c$ is $\Gamma$-invariant. This follows from a more general statement:

\begin{lemma}
Assume that  $(N^d,g_N)$ is a complete simply connected Riemannian manifold of dimension $d\ge 1$, $c:\R\to\R$ is a smooth function and $\Gamma$ is a co-compact group acting totally discontinuously by isometries on the Riemannian manifold $(\R^2\times N,\d s^2+\d t^2+e^{2c(t)}g_N)$. Assume moreover that $\Gamma$ preserves the vector fields $\partial_s$ and $\partial_t$.
Then the function $c$ is invariant by $\Gamma$.
\end{lemma}
\begin{proof}
The last assumption shows that every element $\gamma\in\Gamma$ has the form $\gamma(s,t,x)=(s+s_\gamma,t+t_\gamma,\psi_\gamma(x))$,
where $s_\gamma$ and $t_\gamma$ are real numbers and $\psi_\gamma$ is a diffeomorphism of $N$. The condition that $\gamma$ is an isometry of the metric $\d s^2+\d t^2+e^{2c(t)}g_N$ reads 
$$e^{2c(t)}g_N(X,Y)=e^{2c(t+t_\gamma)}g_N((\psi_\gamma)_*(X),(\psi_\gamma)_*(Y)),\qquad\forall\ t\in\R,\ X,Y\in\T N.$$
Thus $\psi_\gamma$ is a homothety of $(N,g_N)$ with ratio 
\be\label{rhog}\rho_\gamma:=e^{c(t)-c(t+t_\gamma)}\ee
(note that, in particular, this expression does not depend on $t$). 

Assume, for a contradiction, that $c$ is not $\Gamma$-invariant. By \eqref{rhog}, there exists $\gamma_0\in\Gamma$ such that $\rho_{\gamma_0}<1$. The map $\psi_{\gamma_0}$ is a contraction of the complete metric space $(N,d_N)$, where $d_N$ is the distance induced by $g_N$. By the Banach fixed point theorem, $\psi_{\gamma_0}$ has a unique fixed point $x_0\in N$ and 
\be\label{lim}\lim_{k\to\infty}\psi_{\gamma_0}^k(x)=x_0,\qquad\forall\ x\in N.\ee
Let $\gamma$ be any element of $\Gamma$. For every integer $k\in\N$ we have 
$$y_k:=(\gamma_0^k\circ \gamma\circ\gamma_0^{-k})(0,0,x_0)=(s_{\gamma},t_{\gamma},\psi_{\gamma_0}^k(\psi_{\gamma}(x_0))),$$
so by \eqref{lim}, the sequence $\{y_k\}$ converges to $(s_{\gamma},t_{\gamma},x_0)=:y_0$. Since the action of $\Gamma$ is totally discontinuous, this implies that $y_k=y_0$ for $k$ sufficiently large, whence $\psi_{\gamma}(x_0)=x_0$ for every $\gamma\in\Gamma$.

Consider now the continuous map $f:\R^2\times N\to\R_+$ defined by $f(s,t,x):=e^{c(t)}d_N(x,x_0)$. Using \eqref{rhog} an immediate induction shows that $$c(nt_{\gamma_0})=c(0)-n\ln(\rho_{\gamma_0}),\qquad\forall\ n\in\Z,$$
thus showing that $c$ is onto on $\R$. In particular, $f$ is onto on $\R_+$.

For every $\gamma\in\Gamma$ we have using \eqref{rhog}:
\beq(\gamma^*f)(s,t,x)&=&f(s+s_\gamma,t+t_\gamma,\psi_\gamma(x))=e^{c(t+t_\gamma)}d_N(\psi_\gamma(x),x_0)\\
&=&e^{c(t+t_\gamma)}d_N(\psi_\gamma(x),\psi_\gamma(x_0))=\rho_\gamma e^{c(t+t_\gamma)}d_N(x,x_0)=e^{c(t)}d_N(x,x_0)\\&=&f(s,t,x).
\eeq
Thus $f$ is $\Gamma$-invariant and induces a continuous map $\tilde f:\Gamma\backslash(\R^2\times N)\to\R$. Since $f$ is onto, $\tilde f$ is also onto, contradicting the fact that the action of $\Gamma$ on $\R^2\times N$ is co-compact.
\end{proof}

Summarizing, we have proved:

\begin{theorem} \label{main} Let $(M,g,J,\theta)$ be a compact lcK manifold of complex dimension $n>2$ admitting a non-trivial parallel vector field $V$. Then the following (exclusive) possibilities occur:
\begin{enumerate}
\item[(i)] The Lee form $\theta$ is a (non-zero) constant multiple of $V^\flat$, so $M$ is a Vaisman lcK manifold.
\item[(ii)] $(M,g,\Omega,\theta)$ is {\em globally} conformally K\"ahler and there exists a complete simply connected K\"ahler manifold $(N,g_N,\Omega_N)$ of real dimension $2n-2$, a smooth real function $c:\R\to\R$ and a discrete co-compact group $\Gamma$ acting freely and totally discontinuously on $\R^2\times N$, preserving the metric $\d s^2+\d t^2+e^{2c(t)}g_N$, the Hermitian $2$-form $\d s\wedge\d t+e^{2c(t)}\Omega_N$ and the vector fields $\partial_s$ and $\partial_t$, such that $M$ is diffeomorphic to $\Gamma\backslash(\R^2\times N)$, and the structure $(g,\Omega,\theta)$ corresponds to $(\d s^2+\d t^2+e^{2c(t)}g_N, \d s\wedge\d t+e^{2c(t)}\Omega_N,\d c)$ through this diffeomorphism.
\end{enumerate}
\end{theorem}

\begin{example} Typically, one can obtain examples of type (ii) by taking $(N,g_N,\Omega_N)$ to be any compact K\"ahler manifold, $c$ any $T$-periodic function, and $\Gamma$ the group of isometries of $(\R^2\times N,\d s^2+\d t^2+e^{2c(t)}g_N)$ generated by the maps $\gamma_1:(s,t,x)\mapsto (s+1,t,x)$ and $\gamma_2:(s,t,x)\mapsto (s,t+T,x)$. 
\end{example}

\bibliographystyle{amsplain}

\begin{thebibliography}{9}

\bibitem{be} F. Belgun, {\sl On the metric structure of non-K\"ahler
    complex surfaces},  Math. Ann.  {\bf 317}  (2000), 1--40. 

\bibitem{bu} N. Buchdahl, {\sl On compact K\"ahler surfaces},
Ann. Inst. Fourier {\bf 49} no. 1 (1999), 287--302.

\bibitem{do} S. Dragomir, L. Ornea, {\it Locally conformal K\"ahler
    geometry}, Progress in Math. {\bf 155}, Birkh\"auser, Boston, Basel, 1998.

\bibitem{gmo} P. Gauduchon, A. Moroianu, L. Ornea, {\sl Compact homogeneous lcK manifolds are Vaisman}, Math. Ann. (2015) doi: 10.1007/s00208-014-1103-x.

\bibitem{la} A. Lamari, {\sl Courants k\"ahl\'eriens et surfaces compactes}, Ann. Inst. Fourier {\bf 49} no. 1 (1999), 263--285.

\bibitem{ov} L. Ornea, M. Verbitsky, {\sl Structure theorem for
    compact Vaisman manifolds}, Math. Res. Lett., {\bf 10} (2003),  
799--805.

\bibitem{va} I. Vaisman, {\sl A survey of generalized Hopf manifolds},  Rend. Sem. Mat. Univ. Politec. Torino 1983, Special Issue (1984), 205--221.

\end{thebibliography}

\end{document}